\documentclass[reqno,11pt,a4paper,final]{amsart}
\usepackage[a4paper,left=40mm,right=40mm,top=35mm,bottom=35mm,marginpar=25mm]{geometry} 

\usepackage{amsmath}
\usepackage{amssymb}
\usepackage{amsthm}

\usepackage{amscd}
\usepackage[ansinew]{inputenc}
\usepackage{bbm}
\usepackage{color}
\usepackage[english=american]{csquotes}
\usepackage{hyperref}
\usepackage{calc}
\usepackage{amsmath}
\usepackage{amssymb}
\usepackage{amsthm}
\usepackage{amscd}
\usepackage{color}
\usepackage[english=american]{csquotes}
\usepackage{calc}
\usepackage{bm}
\usepackage{dsfont}
\usepackage{enumerate}
\usepackage{mathtools}
\usepackage[dvipsnames]{xcolor}
\usepackage{enumerate}

\numberwithin{equation}{section}

\newtheoremstyle{thmlemcorr}{10pt}{10pt}{\itshape}{}{\bfseries}{.}{10pt}{{\thmname{#1}\thmnumber{ #2}\thmnote{ (#3)}}}
\newtheoremstyle{thmlemcorr*}{10pt}{10pt}{\itshape}{}{\bfseries}{.}\newline{{\thmname{#1}\thmnumber{ #2}\thmnote{ (#3)}}}
\newtheoremstyle{defi}{10pt}{10pt}{\itshape}{}{\bfseries}{.}{10pt}{{\thmname{#1}\thmnumber{ #2}\thmnote{ (#3)}}}
\newtheoremstyle{remexample}{10pt}{10pt}{}{}{\bfseries}{.}{10pt}{{\thmname{#1}\thmnumber{ #2}\thmnote{ (#3)}}}
\newtheoremstyle{ass}{10pt}{10pt}{}{}{\bfseries}{.}{10pt}{{\thmname{#1}\thmnumber{ #2}\thmnote{ (#3)}}}

\theoremstyle{thmlemcorr}
\newtheorem{theorem}{Theorem}
\numberwithin{theorem}{section}
\newtheorem{lemma}[theorem]{Lemma}
\newtheorem{corollary}[theorem]{Corollary}
\newtheorem{proposition}[theorem]{Proposition}

\newtheorem{conjecture}[theorem]{Conjecture}

\theoremstyle{thmlemcorr*}
\newtheorem{theorem*}{Theorem}
\newtheorem{lemma*}[theorem]{Lemma}
\newtheorem{corollary*}[theorem]{Corollary}
\newtheorem{proposition*}[theorem]{Proposition}
\newtheorem{problem*}[theorem]{Problem}
\newtheorem{conjecture*}[theorem]{Conjecture}

\theoremstyle{defi}
\newtheorem{definition}[theorem]{Definition}

\theoremstyle{remexample}

\theoremstyle{ass}

\newcommand{\Crm}{\mathrm{C}}

\newcommand{\Lrm}{\mathrm{L}}

\newcommand{\Wrm}{\mathrm{W}}

\newcommand{\Acal}{\mathcal{A}}

\newcommand{\Fcal}{\mathcal{F}}

\newcommand{\Hcal}{\mathcal{H}}
\newcommand{\Ical}{\mathcal{I}}

\newcommand{\Mcal}{\mathcal{M}}
\newcommand{\Ncal}{\mathcal{N}}

\newcommand{\Abb}{\mathbb{A}}
\renewcommand{\Bbb}{\mathbb{B}}

\newcommand{\Sbb}{\mathbb{S}}

\DeclareMathOperator{\curl}{curl}

\DeclareMathOperator{\rank}{rank}

\newcommand{\setn}[2]{\{\, #1 \ \ \textup{\textbf{:}}\ \ #2 \,\}}
\newcommand{\setb}[2]{\bigl\{\, #1 \ \ \textup{\textbf{:}}\ \ #2 \,\bigr\}}

\newcommand{\setBB}[2]{\biggl\{\, #1 \ \ \textup{\textbf{:}}\ \ #2 \,\biggr\}}

\newcommand{\abs}[1]{|#1|}

\newcommand{\absb}[1]{\bigl|#1\bigr|}

\newcommand{\absBB}[1]{\biggl|#1\biggr|}

\newcommand{\di}{\mathrm{d}}
\newcommand{\dd}{\;\mathrm{d}}

\newcommand{\N}{\mathbb{N}}
\newcommand{\R}{\mathbb{R}}

\newcommand{\loc}{\mathrm{loc}}

\newcommand{\toweakstar}{\overset{*}\rightharpoonup}

\newcommand{\todown}{\downarrow}

\newcommand{\conv}{\!\star\!}

\newcommand{\BV}{\mathrm{BV}}
\newcommand{\BD}{\mathrm{BD}}

\newcommand{\Gr}{\mathrm{Gr}}

\newcommand{\term}[1]{\emph{#1}}

\def\Xint#1{\mathchoice 
{\XXint\displaystyle\textstyle{#1}}%
{\XXint\textstyle\scriptstyle{#1}}%
{\XXint\scriptstyle\scriptscriptstyle{#1}}%
{\XXint\scriptscriptstyle\scriptscriptstyle{#1}}%
\!\int} 
\def\XXint#1#2#3{{\setbox0=\hbox{$#1{#2#3}{\int}$} 
\vcenter{\hbox{$#2#3$}}\kern-.5\wd0}} 
\def\dashint{\,\Xint-}

\newcommand{\restrict}{\begin{picture}(10,8)\put(2,0){\line(0,1){7}}\put(1.8,0){\line(1,0){7}}\end{picture}}

\newcommand{\eps}{\varepsilon}
\renewcommand{\phi}{\varphi}
\renewcommand{\hat}{\widehat}

\newcommand{\A}{\mathcal{A}}
\newcommand{\B}{\mathcal{B}}
\newcommand{\bA}{\mathbb{A}}
\newcommand{\LL}{\mathcal{L}}
\newcommand{\HH}{\mathcal{H}}

\renewcommand{\L}{\mathrm{L}}

\newcommand{\I}{\mathcal{I}}
\newcommand{\p}{\mathcal{P}}

\newcommand{\bB}{\mathbb{B}}

\DeclareMathOperator{\Div}{div}
\DeclareMathOperator{\spt}{spt}

\def\d{{\rm  d}}
\def \p{\boldsymbol{p}}
\DeclareMathOperator{\Id}{\mathrm{Id}}
\def\BV{\textrm{BV}}
\def\BD{\textrm{BD}}

\newcommand{\res}{\restrict}

\title[Dimensional estimates and rectifiability]{Dimensional estimates and rectifiability for measures satisfying linear PDE constraints}

\author[A.~Arroyo-Rabasa]{Adolfo Arroyo-Rabasa}
\address{A.A.-R.: Mathematics Institute, University of Warwick, Coventry CV4 7AL, UK.}
\email{Adolfo.Arroyo-Rabasa@warwick.ac.uk}

\author[G.~De~Philippis]{Guido De Philippis}
\address{G.D.P.: SISSA, Via Bonomea 265, 34136 Trieste, Italy}
\email{guido.dephilippis@sissa.it}

\author[J.~Hirsch]{Jonas Hirsch}
\address{J.H.:  Mathematisches Institut, Universit\"at Leipzig, Augustus Platz 10, D04109 Leipzig, Germany}
\email{hirsch@math.uni-leipzig.de}

\author[F.~Rindler]{Filip Rindler}
\address{F.R.: Mathematics Institute, University of Warwick, Coventry CV4 7AL, UK, and The Alan Turing Institute, British Library, 96 Euston Road, London NW1 2DB London, UK.}
\email{F.Rindler@warwick.ac.uk}

\begin{document}

\begin{abstract}
We establish the rectifiability of measures satisfying a linear PDE constraint. The obtained rectifiability dimensions are optimal for many usual PDE operators, including all first-order systems and all second-order scalar operators. In particular, our general theorem provides a new proof of the rectifiability results for functions of bounded variations (BV) and functions of bounded deformation (BD). For divergence-free tensors we obtain refinements and new proofs of several known results on the rectifiability of varifolds and defect measures.

\vspace{4pt}

%


\vspace{4pt}

\noindent\textsc{Date:} \today{}.
\end{abstract}

\maketitle

\section{Introduction}

Let \(\A\) be a    $k^{\text{\tiny th}}$-order linear constant-coefficient PDE operator acting  on \(\mathbb R^m\)-valued functions on $\R^d$ via
\[
  \A \varphi := \sum_{|\alpha|\le k}A_{\alpha} \partial^\alpha \phi\qquad\textrm{for all \(\varphi\in \Crm^\infty(\R^d;\mathbb R^m)\),}
\]
where  $A_{\alpha}\in \mathbb R^{n}\otimes \R^m $ ($\cong \R^{n \times m}$) are (constant) matrices, \(\alpha=(\alpha_1,\dots,\alpha_d)\in (\mathbb N \cup \{0\})^d\) is a multi-index and   $\partial^\alpha:=\partial_1^{\alpha_1}\ldots\partial_d^{\alpha_d}$. We also assume that at least one $A_\alpha$ with $\abs{\alpha} = k$ is non-zero.

An \(\R^m\)-valued Radon measure  \(\mu \in \mathcal M(U;\R^m)\) defined on an open set \(U\subset \R^d \)  is said to be \term{\(\A\)-free} if 
\begin{equation}\label{e:operator A}
\A \mu=0   \qquad\text{in the sense of distributions on $U$.}
\end{equation}
The Lebesgue--Radon--Nikod\'{y}m theorem implies that 
\[
\mu=g\LL^d+ \frac{\d\mu}{\d|\mu|}|\mu|^s,
\]
where \(g\in \L^1(U;\R^m)\),   \(|\mu|^s\) is the singular part of  the total variation measure \(|\mu|\) with respect to the $d$-dimensional Lebesgue measure \(\LL^d\),  and  
\[
\frac{\d\mu}{\d|\mu|}(x):=\lim_{r\to 0} \frac{\mu(B_r(x))}{|\mu|(B_r(x))}
\]
is the \term{polar} of \(\mu\), which exists and belongs to \(\mathbb S^{m-1}\) for \(|\mu|\)-almost every \(x \in U\).

In~\cite{De-PhilippisRindler16} it was shown that for any \(\A\)-free measure there is a strong constraint on the directions of the polar at singular points:

\newpage

\begin{theorem}[{\cite[Theorem~1.1]{De-PhilippisRindler16}}]\label{DR}
Let \(U\subset \R^d\) be an open set, let \(\A\) be a $k^{\text{\tiny th}}$-order linear constant-coefficient differential operator as above, and let \(\mu \in \mathcal M(U;\R^m)\) be an \(\A\)-free Radon measure on $U$ with values in $\R^m$. Then,
\[
\frac{\d\mu}{\d|\mu|}(x)\in \Lambda_\A \qquad\textrm{for \(|\mu|^s\)-a.e.\ \(x\in U\),}
\]
where \(\Lambda_\A\) is the {\em wave cone} associated to \(\A\), namely
\begin{equation}\label{e:wc}
\Lambda_{\A}:=\bigcup_{\xi \in \R^d\setminus\{0\}}\ker \bA^k(\xi),\qquad \bA^k(\xi):=\sum_{|\alpha|=k} A_{\alpha} \xi^\alpha.
\end{equation}
\end{theorem}
 
It has been shown in~\cite{De-PhilippisRindler16}, see also~\cite{De-Philippis16,De-PhilippisRindler17a} for recent surveys and~\cite[Chapter~10]{Rindler18book} for further explanation, that by suitably choosing the operator \(\A\), the study of the singular part of \(\A\)-free measures has several consequences in the calculus of variations and in geometric measure theory. In particular, we recall the following:
\begin{itemize}
\item If \(\A={\rm curl}\), the above theorem gives a new proof of {Alberti}'s rank-one theorem~\cite{Alberti93} (see also~\cite{MassaccesiVittone16} for a different proof based on a geometrical argument).
\item If \(\A=\Div\), combining Theorem~\ref{DR} with the result of~\cite{AlbertiMarchese16}, one obtains the weak converse of Rademacher's theorem (see~\cite{De-PhilippisMarcheseRindler17,GigliPasqualetto16,KellMondino16} for other consequences in metric geometry).
\end{itemize}

The main results of this paper is to show how Theorem \ref{DR}  can be improved by further constraining the direction of the polars on ``lower dimensional parts'' of the measure \(\mu\) and to establish some   consequences of this fact concerning dimensional estimates and rectifiability of  \(\A\)-free measures.
 To this end let us define a hierarchy of wave cones as follows:

\begin{definition}[$\ell$-wave cone] Let \(\Gr(\ell,d)\) be the Grassmannian of \(\ell\)-planes in \(\R^d\).
For $\ell =1,\ldots,d$ we define the \emph{$\ell$-dimensional wave cone} as 
\[
	\Lambda^\ell_\A:=\bigcap_{\pi \in \Gr(\ell,d)} \bigcup_{\substack{\xi \in \pi\setminus\{0\} }} \ker\bA^k(\xi),
\]
where \(\bA^k(\xi)\) is defined as in~\eqref{e:wc}.	
\end{definition}

Equivalently, $\Lambda^\ell_{\A}$ can be defined by the following analytical property:
\[
\text{$\lambda \notin \Lambda^\ell_\A \quad \Longleftrightarrow \quad (\A\restrict \pi) \lambda$ is elliptic for some $\pi \in \Gr(\ell,d)$,} 
\] 
where $(\A \restrict \pi)$ is the partial differential operator
 \[
 \Crm^\infty(\pi;\R^m)\ni\phi \mapsto (\A \restrict \pi)(\phi):= \A(\phi \circ \p_{\pi}), 
 \]
with \(\p_{\pi}\) the orthogonal projection onto \(\pi\). 

Note that, by the very definition of  \(\Lambda^{\ell}_\A\), we have the following inclusions:
\begin{equation}\label{e:inc}
\Lambda^{1}_{\A}=\bigcap_{\xi\in \R^d\setminus \{0\}} \ker \bA^{k}(\xi)\subset\Lambda_{\A}^j \subset\Lambda_{\A}^\ell\subset \Lambda_{\A}^d=\Lambda_{\A},\qquad 1\le j\le \ell\le d. 
\end{equation}

To state our main theorem, we also recall the definition of the integral-geometric measure, see~\cite[Section~5.14]{Mattila95}: Let \(\ell \in \{0,\ldots,d\}\). For a Borel  set \(E\subset \R^d\), the \term{\(\ell\)-dimensional integral-geometric (outer) measure} is 
\[
\mathcal I^\ell(E):=\int_{\Gr(\ell,d)} \int_{\pi}\HH^{0}(E\cap \p_{\pi}^{-1}(x)) \dd \HH^\ell(x) \dd \gamma_{\ell,d}(\pi),
\]
where \(\gamma_{\ell,d}\) is the {unique} \(\mathrm{O}(d)\)-invariant probability measure on \(\Gr(\ell,d)\) and $\HH^\ell$ is the $\ell$-dimensional Hausdorff measure (normalized as in~\cite{Mattila95} such that $\HH^\ell(B^\ell_1) = 2^\ell$, where $B^\ell_1$ is the $\ell$-dimensional unit ball).

Our main result establishes that the polar of an \(\A\)-free measure is constrained to lie in a smaller cone on \(\I^\ell\)-null sets: 

\begin{theorem}\label{t:unsatisfactory}
	Let $U \subset \R^d$ be open, let $\A$ be as in~\eqref{e:operator A}, and let $\mu \in \mathcal{M}(U;\R^m)$ be an $\A$-free
	 measure on $U$. If  \(E \subset U\) is a Borel set with \(\mathcal I^{\ell}(E) = 0\) for some \(\ell \in \{0,\ldots,d\}\), then 
\[
		\frac{\d \mu}{\d|\mu|}(x) \in \Lambda^\ell_\A\qquad\mbox{for \(|\mu|\)-a.e. \(x\in E\)}.
\]
\end{theorem}
Note that, by taking \(\ell=d\),  Theorem~\ref{t:unsatisfactory} recovers Theorem~\ref{DR}. As a corollary we obtain the following dimensional estimates on \(\A\)-free measures; see also~\cite{Arroyo-Rabasa18} for a different proof of~\eqref{e:dim} in the case of first-order systems.

\begin{corollary}[dimensional estimate]\label{cor:aac}
Let \(\A\) and \(\mu\) be as in Theorem~\ref{t:unsatisfactory} and assume that \(\Lambda^\ell_{\A}=\{0\}\) for some \(\ell \in \{0,\ldots,d\}\).  Then,
\[
   \text{$E \subset U$ Borel with $\I^{\ell}(E)=0$}  \quad\Longrightarrow\quad
   |\mu|(E)=0.
\]
In particular,
\[
  \mu \ll \I^\ell\ll\HH^\ell
\]
and thus 
\begin{equation}\label{e:dim}
\dim_{\HH} \mu:=\sup \, \setb{ \ell }{ \mu\ll\HH^\ell } \ge \ell_\Acal,
\end{equation}
where
\begin{equation}\label{ell1}
  \ell_\Acal : =\max\setb{ \ell }{ \Lambda_{\A}^\ell=\{0\}}.
\end{equation}
\end{corollary}

The results above and~\eqref{e:inc} entail that the smaller the dimension of an \(\A\)-free  measure \(\mu\) is, the  more its polar is constrained at singular points. Let us also remark that the $1$-dimensional wave cone \(\Lambda^{1}_{\A}\) has been implicitly introduced by \textsc{van~Schaftigen} in~\cite{Van-Schaftingen13}. There, the author calls a (homogeneous) oparator $\A$ \emph{cocanceling} provided that $\Lambda_\A^1 = \{0\}$. Moreover, it is shown that the cocanceling condition is equivalent to the property
	\[
		\A(\lambda \delta_0) = 0 \;\;\text{for some}\;\; \lambda \in \R^m \qquad \Longrightarrow \qquad \lambda = 0.
	\] 
Thus, the conclusion of Theorem~\ref{t:unsatisfactory} improves upon the dimensional estimates for $\A$-free measures with $\A$ cocanceling.

The use of the integral-geometric measure, besides being natural in the proof, allows one to use the Besicovitch--Federer rectifiability criterion to deduce the following rectifiability result. Recall that for a  positive measure \(\sigma\in \mathcal M_{+}(U)\)  the  \emph{upper $\ell$-dimensional  density} is defined as
	\[
		\theta^*_\ell(\sigma)(x) \coloneqq \limsup_{r \to 0} \frac{\sigma(B_r(x))}{(2r)^\ell} = \limsup_{r \to 0} \frac{\sigma(B_r(x))}{\HH^\ell(B^\ell_r)}, \qquad x \in U.
	\]  

\begin{theorem}[rectifiability]\label{t:rect}
Let \(\A\) and \(\mu\) be as in Theorem~\ref{t:unsatisfactory}, and assume that \(\Lambda^{\ell}_{\A}=\{0\}\). Then, the set $\{\theta^*_\ell(|\mu|) = +\infty\}$ is $|\mu|$-negligible. 
Moreover, \(\mu\res \{\theta^{*}_{\ell}(|\mu|)>0\}\) is concentrated on an \(\ell\)-rectifiable set \(R\) and
\[
\mu\res R=\theta^*_{\ell}(|\mu|)\lambda\,\HH^\ell\res R,
\]
where \(\lambda \colon R \to \mathbb S^{m-1}\) is $\HH^\ell$-measurable; for \(\HH^\ell\)-almost every \(x_0\in R\) (or, equivalently, for \(|\mu|\)-almost every \(x_0\in R\)),  
\begin{equation}\label{bu}
(2r)^{-\ell}(T^{x_0,r})_\#\mu\overset{*}{\rightharpoonup} \theta^{*}_{\ell}(|\mu|)(x_0) \lambda(x_0)\HH^\ell\res (T_{x_0} R)  \qquad
\text{as $r \todown 0$;}
\end{equation}
and
\begin{equation}\label{kernel}
\lambda(x_0)\in \bigcap_{\xi \in (T_{x_0} R)^\perp} \ker \bA^k(\xi).
\end{equation}
Here  $T^{x_0,r}(x):=(x-x_0)/r$ and \(T_{x_0} R\) is the the approximate tangent plane to \(R\) at \(x_0\).
\end{theorem}

Theorem~\ref{t:rect}  contains the classical rectfiability result for the jump part of the gradient of a \({\rm BV}\) function, see~\cite{AmbrosioFuscoPallara00}, and the analogous  result for \(\BD\), see~\cite{Kohn79,AmbrosioCosciaDalMaso97}. By choosing \(\A=\Div\) we also recover and (in some cases slightly generalize) several known rectifiability criteria, such as {Allard's} rectifiability theorem for varifolds~\cite{Allard72},  its recent extensions to anisotropic energies~\cite{De-PhilippisDe-RosaGhiraldin18}, the rectifiability of generalized varifolds established in~\cite{AmbrosioSoner97}, and the rectifiability of various defect measures  in the spirit of~\cite{Lin99}, see also~\cite{Moser03}. We refer the reader to Section~\ref{sec:applications} for some of these statements.

It is worth noting that, with the exception of the $\BD$-rectifiability result in~\cite[Proposition~3.5]{AmbrosioCosciaDalMaso97}, none of the above rectifiability criteria rely on the Besicovitch--Federer theorem and their proofs are based on more standard blow-up techniques. However, in the generality of Theorem~\ref{t:rect} a blow-up proof seems hard to obtain. Indeed, roughly,  a  blow-up argument follows two steps:

\begin{itemize}
\item By some measure-theoretic arguments one shows that, up to a subsequence,
\[
r^{-\ell} T^{x_0,r}\mu\overset{*}{\rightharpoonup}\lambda \sigma
\]
for some positive measure \(\sigma\) and some fixed vector \(\lambda\).
\item One exploits this information together with the \(\A^k\)-freeness of \(\lambda \sigma\), where $\A^k$ is the principal part of $\A$, to deduce that \(\sigma\) is translation-invariant along the directions in an \(\ell\)-dimensional plane \(\pi\) and thus \(\sigma=\HH^{\ell}\res \pi\). In this step one usually uses that \(\pi\) is uniquely determined by \(\lambda\) and \(\A\).
\end{itemize} 
However, assuming that \(\sigma=\HH^{\ell}\res \pi\), the only information one can get is 
\[
\lambda \in \bigcap_{\xi \in \pi^\perp} \ker \bA^k(\xi),
\]
 see Lemma \ref{lm:rec}, and this does not uniquely determine \(\pi\) in general.

Let us now briefly discuss the optimality of our results. First note that \eqref{bu} and \eqref{kernel} are true whenever an \(\A\)-free measure \(\mu\)  has a non-trivial part concentrated on an \(\ell\)-rectifiable set \(R\), see Lemma \ref{lm:rec} below.

In particular, defining for $\ell = 0,\ldots,d-1$ the cone
\[
\mathcal{N}_{\A}^\ell:= \bigcup_{\pi \in \Gr(\ell,d)} \bigcap_{\xi \in \pi^\perp} \ker \bA^k(\xi)= \bigcup_{\tilde\pi \in \Gr(d-\ell,d)} \bigcap_{\xi \in \tilde{\pi}} \ker \bA^k(\xi),
\]
we have  that
\[
\Lambda_{\A}^{1}=\mathcal{N}_{\A}^0\subset \mathcal{N}_{\A}^\ell \subset \mathcal{N}_{\A}^j\subset  \mathcal{N}_{\A}^{d-1}=\Lambda_{\A},\qquad
0\le \ell\le j\le d-1,
\]
and 
\[
\mathcal{N}_{\A}^\ell\subset \Lambda_{\A}^{\ell+1}, \qquad 0 \le \ell  \le d -2.
\]
Hence, setting
\begin{equation}\label{ell2}
\ell_\Acal^*:=\min\,\setb{\ell }{ \mathcal{N}_{\A}^\ell \ne \{0\}},
\end{equation}
 the above discussion yields that if \(\mu\) has a non-trivial \(\ell\)-rectifiable part, then necessarily
\begin{equation*}
\ell\ge\ell_\Acal^*,
\end{equation*}
and this bound is sharp for homogeneous operators since if \(\lambda \in \bigcap_{\xi \in \pi^\perp} \ker \bA^k(\xi)\setminus\{0\}\) for some \(\ell\)-plane \(\pi\), then \(\lambda \HH^{\ell}\res \pi\) is an \(\A^k\)-free measure.

Recalling  the definition of $\ell_\Acal$ in~\eqref{ell1}, this discussion together with Corollary~\ref{cor:aac} and~\eqref{ell2} can then be summarized for homogeneous operators $\A$ as
\[
\ell_\Acal \le \min \, \setb{ \dim_{\HH} \mu }{ \textrm{\(\mu\) is \(\A\)-free} } \le \ell_\Acal^*.
\]

For first-order operators it is not hard to check that \(\ell_\Acal=\ell_\Acal^*\) (by the linearity of $\xi \mapsto \Abb^k(\xi)$). The same is true for second-order \emph{scalar} operators ($n=1$) by reducing the polynomial to canonical form (which makes $\Abb^k(\xi)$ linear in $\xi_1^2,\ldots,\xi_d^2$). Hence, the above inequality for such homogeneous operators becomes an equality and our theorem is sharp.

On the other hand, it is easy to build examples where \(\ell_\Acal<\ell_\Acal^*\). For instance, one can easily check that for the $3^\text{\tiny{rd}}$-order scalar operator defined on \(\Crm^{\infty}(\R^{3})\) by 
\[
  \A := \partial_{x_{1}}^3 + \partial_{x_{2}}^3 + \partial_{x_{3}}^3 
  \]
we have $\ell_\A = 1 < 2 = \ell^{*}_\A$ since its characteristic set $\setn{ \xi \in \R^3 }{ \xi_1^3 + \xi_2^3 + \xi_3^3=0 }$ is a ruled surface  (and hence it contains lines) but it does  not contain  planes. Moreover, let $\widetilde{\A}$ be the $6^\text{\tiny{th}}$-order operator acting on maps from $\R^3$ to $\R^2$ with symbol
\[
  \widetilde{\Abb}(\xi)\begin{pmatrix}w_1\\w_2\end{pmatrix} := (\xi_1^6+\xi_2^6+\xi_3^6)w_1 + (\xi_1^3 + \xi_2^3 + \xi_3^3)^2w_2,  \qquad \xi \in \R^3.
\]
For this operator we still have $\ell_{\widetilde{\A}} = 1 < 2 = \ell^{*}_{\widetilde{\A}}$, but $\widetilde{A}$ additionally satisfies Murat's constant rank condition~\cite{Murat81}.

Let us remark that in the case  \(\ell_\Acal<\ell_\Acal^*\), Theorem~\ref{t:rect} implies that if \(\mu\) is an \(\A\)-free measure, then 
\[
 \qquad|\mu|\big(\{\theta^{*}_{\ell_\Acal}(|\mu|)>0\})=0.
\] 
Hence  \(\mu\) is ``more diffuse'' than an \(\ell_\Acal\)-dimensional measure. Furthermore, $\mu$ cannot sit on rectifiable sets of any (integer) dimension $\ell \in [\ell_\A,\ell_\A^*)$. It seems thus reasonable to expect that its dimension should be larger than \(\ell_\Acal\).  In particular, one might conjecture the following improvement of Corollary~\ref{cor:aac}:

\begin{conjecture}
Let \(\mu\) be \(\A\)-free and let \(\ell_\Acal^*\) be the rectifiability dimension defined in~\eqref{ell2}. Then,
\[
\dim_{\HH} \mu\ge \ell_\Acal^*.
\]
\end{conjecture}
We note that the same conjecture has also  been advanced by \textsc{Raita} in~\cite[Question~5.11]{Raita18TR}; also see~\cite[Conjecture~1.5]{AyoushWojciechowski17?}.

Further, if one extends \textsc{van~Schaftigen}'s terminology~\cite{Van-Schaftingen13} by saying that $\A$ is ``$\ell$-cocanceling'' provided that $\Ncal^{\ell-1}_\A = \{0\}$ (classical cocanceling then being $1$-cocanceling while ellipticity is \(d\)-cocancelling), the above conjecture  reads as
\[
  \text{$\A$ $\ell$-cocanceling, $\Acal \mu = 0$}  \qquad \Longrightarrow \qquad 
  \text{$\mu \ll\HH^{\ell}$.}
\]
Recently, a related (dual) notion of ``$\ell$-canceling'' operators has been introduced in~\cite{SpectorVanSchaftingen18?}.

We conclude this introduction by remarking that the above results can be used to provide dimensional estimates and rectifiability results for measures whose decomposability bundle, defined  in~\cite{AlbertiMarchese16}, has dimension at least \(\ell\). Namely, in this case the measure is absolutely continuous with respect to \(\I^{\ell}\) and the set where the upper \(\ell\)-dimensional density is positive, is rectifiable, compare with~\cite[Theorem 2.19]{Bate17} and with~\cite{AlbertiMassaccesiStepanov}. However, since by its very definition  the dimension of the decomposability bundle is stable under projections,  in this setting one can directly rely on~\cite[Corollary 1.12]{De-PhilippisRindler16}. This is essentially the strategy followed in the cited  references.

 \subsection*{Acknowledgments} 

This project has received funding from the European Research Council (ERC) under the European Union's Horizon 2020 research and innovation programe, grant agreement No 757254 (SINGULARITY), and from the INDAM-grant ``Geometric Variational Problems". We thank the anonymous referee for various comments that improved the presentation of this paper.

\section{Proofs}\label{sec:proofs}

The proof  of Theorem~~\ref{t:unsatisfactory} is a combination of ideas from~\cite{De-PhilippisRindler16} and~\cite{De-PhilippisDe-RosaGhiraldin18}. We start with the following lemma.

\begin{lemma}\label{l:strong constancy lemma}
Let $\B$ be a homogeneous $k^\text{\tiny{th}}$-order linear constant-coefficient operator on $\R^\ell$,
\[
\B:=\sum_{|\beta|=k} A_\beta\partial^\beta, \qquad A_\beta\in \R^n\otimes \R^m, \qquad \beta\in (\N \cup \{0\})^{\ell}.
\]
Let   $\{\nu_j\}\subset \mathcal{M}(B_1^\ell;\R^m)$, where \(B_1^\ell\subset \R^\ell\) is the unit ball in \(\R^\ell\), be a uniformly norm-bounded sequence of Radon measures satisfying the following assumptions:
\begin{itemize}
\item[(a1)] $\B \lambda $ is elliptic for some $\lambda  \in \R^m$, that is, 
\[
 \lambda \notin \ker \bB(\xi) \quad\text{for all $\xi \in \R^\ell \setminus \{0\}$},
\]
where $\bB(\xi):=\sum_{|\beta|=k} A_{\beta} \xi^\beta \in \R^n \otimes \R^m$;
\item[(a2)] $\{(\Id-\Delta)^{-\frac{s}{2}}\B\nu_j\}_j$ is pre-compact in $\L^1(B_1^\ell;\R^n)$ for some \(s<k\);
\item[(a3)] $\displaystyle \lim_{j\to \infty} \int_{B^\ell_1} \biggl| \frac{\d\nu_j}{\d|\nu_j|} - \lambda\biggr| \dd|\nu_j|= 0$.
\end{itemize}
Then, up to taking a subsequence, there exists $\theta \in \L^1(B^\ell_1)$ such that 
\begin{equation}\label{e:strong convergence}
\big||\nu_j|- \theta \LL^\ell \big|(B^\ell_t)\to 0 \quad\text{for all $0 < t < 1$.}	
\end{equation}
\end{lemma}

\begin{proof}
The proof is a straightforward modification of the main step of the proof of~\cite[Theorem~1.1]{De-PhilippisRindler16}, see also~\cite{Allard86} and~\cite[Chapter~10]{Rindler18book}. We give it here in terse form for the sake of completeness.

Passing to a subsequence we may assume that $\abs{\nu_j} \overset{*}{\rightharpoonup} \sigma$ in $ \Crm^\infty_c(B_1^\ell)^*$ for some positive measure $\sigma \in \mathcal{M}^+(B_1^\ell)$.	We must show that $\sigma = \theta \LL^d$ and that \eqref{e:strong convergence} holds. 
Fix $t<1$ and two smooth cut-off functions $0 \le \chi \le \tilde{\chi} \le 1$ with $\chi=1$ on $B_t$, $\tilde{\chi}=1$ on $\spt(\chi)$, and $\spt(\chi)\subset \spt(\tilde{\chi}) \subset B_1$. 
Let $(\phi_{\eps})_{\eps > 0}$ be a family of smooth approximations of the identity.  
Choose $\epsilon_j \todown 0$ with $0 < \epsilon_j < 1-t$ for all $j$, such that 
\[ \absb{\abs{\nu_j}-\sigma}(B_t) \le \absb{\varphi_{\epsilon_j}\conv\abs{\nu_j}-\sigma}(B_t)+ 2^{-j}.\]
We will show that the sequence
\[ u_j:= \chi \, (\varphi_{\epsilon_j}\conv\abs{\nu_j}) \]
is pre-compact in $\Lrm^1(B_1)$, which proves the lemma.

For every $j$ we set $f_j := \B\nu_j$ and compute
\begin{align*}
	\mathcal{B}(\lambda u_j) &=\phantom{:} \chi \, \mathcal{B}\biggl[ \varphi_{\epsilon_j} \conv\left( \Bigl( \lambda-\textstyle\frac{\di\nu_j}{\di\abs{\nu_j}} \Bigr) \abs{\nu_j}\right)\biggr]  + \chi \, (\varphi_{\epsilon_j}\conv f_j)+ [\mathcal{B},\chi](\lambda \, \varphi_{\epsilon_j}\conv \abs{\nu_j})\\
	&=\phantom{:} \mathcal{B} \biggl[ \chi \,  \varphi_{\epsilon_j} \conv\left( \Bigl( \lambda-\textstyle\frac{\di\nu_j}{\di\abs{\nu_j}} \Bigr) \abs{\nu_j}\right) \biggr] + \chi \, (\varphi_{\epsilon_j}\conv f_j)+ [\mathcal{B},\chi](\tilde{\chi} \, \varphi_{\epsilon_j}\conv \nu_j)\\
	 &=: \mathcal{B} V_j + \chi \, (\varphi_{\epsilon_j}\conv f_j) + [\mathcal{B},\chi]W_j.
\end{align*}
Note that the commutator $[\mathcal{B},\chi] := \mathcal{B} \circ \chi - \chi \circ \mathcal{B}$ is a differential operator of order at most $k-1$ with smooth coefficients.
Taking the Fourier transform (which we denote by $\mathcal{F}$ or by the hat~``$\hat{\phantom{w}}$''), multiplying by $[\mathbb{B}(\xi)\lambda]^*$, and adding $\hat{u}_j(\xi)$,  we obtain 
\begin{align*}
 (1+\abs{\mathbb{B}\lambda}^2) \hat{u}_j = [\mathbb{B}\lambda]^*\mathbb{B}\hat{V}_j + [\mathbb{B}\lambda]^* 
 \Fcal[\chi \, (\varphi_{\epsilon_j}\conv f_j)]  + [\mathbb{B}\lambda]^*\Fcal[[\mathcal{B},\chi]W_j]+\hat{u}_j.
\end{align*}
Hence, 
\[
	u_j = T_0[V_j] + T_1[\chi \, (\varphi_{\epsilon_j}\conv f_j)] + T_2[W_j]+T_3[u_j]
\]
with the pseudo-differential operators $T_0,\ldots,T_3$ defined as follows:
\begin{align*}
T_0[V]&:= \mathcal{F}^{-1}\left[ \frac{[\mathbb{B}\lambda]^*\mathbb{B}}{1+\abs{\mathbb{B}\lambda}^2} \, \hat{V}\right],\\
T_1[f]&:= \mathcal{F}^{-1}\left[ \frac{[\mathbb{B}\lambda]^*}{1+\abs{\mathbb{B}\lambda}^2} \, \hat{f}\right],\\
T_2[W]&:=\mathcal{F}^{-1}\left[ \frac{[\mathbb{B}\lambda]^*}{1+\abs{\mathbb{B}\lambda}^2} \, \mathcal{F}[{[\mathcal{B},\chi]W}]\right],\\
T_3[u]&:=\mathcal{F}^{-1}\left[ \frac{1}{1+\abs{\mathbb{B}\lambda}^2} \hat u\right].
\end{align*}
We see that, in the language of pseudo-differential operators (see for instance~\cite[Chapter~VI]{Stein93}):
\begin{itemize}
\item[(i)] the symbol for $T_0$ is a H\"ormander--Mihlin multiplier (i.e.\ a pseudo-differential operator  with smooth symbol of order $0$) since, due to~(a1), $\abs{\Bbb(\xi)\lambda} \geq c\abs{\xi}^k$ for some $c > 0$ and all $\xi \in \R^\ell$;
\item[(ii)]  $T_1$ is a  pseudo-differential operator  with smooth symbol of order \(-k\);
\item[(iii)] $T_2$ is a pseudo-differential operator  with smooth symbol of order \(-1\); 
\item[(iv)] $T_3$ is a pseudo-differential operator  with smooth symbol of order \(-2k\).
\end{itemize}
By the  classical theory of Fourier multipliers and pseudo-differential operators we then get the following:
\begin{itemize}
\item[(I)] $T_0$ is bounded from $\Lrm^1$ to $\Lrm^{1, \infty}$ (weak-$\Lrm^1$), see e.g.~\cite[Theorem~6.2.7]{Grafakos14}. Owing to (a3), it follows that for $j \to \infty$ we obtain
\begin{align*}
  \int \abs{V_j} \dd x &\le\int \chi \, \varphi_{\epsilon_j}\conv\left(\absBB{ \frac{\di\nu_j}{\di\abs{\nu_j}} - \lambda} \, \abs{\nu_j}\right) \dd x \\ 
  &\le \int_{B_1}\absBB{ \frac{\di\nu_j}{\di\abs{\nu_j}} - \lambda} \dd\abs{\nu_j} \\
  &\to 0.
\end{align*}
Thus,
\[
	\sup_{t \ge 0} t \, \LL^d(\{ \abs{T_0[V_j]}>t\}) \le C \int \abs{V_j} \dd x \to 0  \qquad\text{as $j \to \infty$.}
\]
That is, $T_0[V_j] \to 0$ in measure.
\item[(II)] Due to (a2), $T_1[f_j]$ is pre-compact in $\Lrm^1$ (this follows directly by the symbolic calculus~\cite[Section~VI.3]{Stein93} or direct manipulation of Fourier multipliers).
\item[(III)] $T_2$ and \(T_3\) are compact operators from $\Lrm^1_c$ to $\Lrm^1_{\text{loc}}$ (see for instance~\cite[Propositions~VI.4,~VI.5]{Stein93} in conjunction with Lemma~10.1 in~\cite{De-PhilippisRindler16} or Lemma~10.11 in~\cite{Rindler18book}) and thus the families $\{T_2[W_j]\}$, \(\{T_3[u_j]\}\) are pre-compact in $\Lrm^1$.
\end{itemize}
Hence, passing to a subsequence, we may assume that 
$ T_1[f_j] + T_2[W_j] +T_3[u_j]\to \theta$ in $\Lrm^1_{\text{loc}}$ and $T_0[V_j] \to 0$ in measure. Since furthermore $u_j \ge 0$, we can apply Lemma~\ref{lem:conv_trick} below and deduce that $T_0[V_j] \to 0$ strongly in $\Lrm^1$. This concludes the proof.
\end{proof}

The following is Lemma~2.2 in~\cite{De-PhilippisRindler16}, we report here  its straightforward  proof for the sake of completeness.
\begin{lemma}\label{lem:conv_trick}
Let $\{f_j\} \subset \Lrm^1(B_1)$ be such that
\begin{enumerate}[(i)]
\item $f_j\toweakstar 0$ in $ \Crm^\infty_c(B_1)^*$;
\item the negative parts $f_j^- := \max\{-f_j,0\}$ of the $f_j$'s converge to zero in measure, i.e.,
 \[
  \lim_{j\to \infty}\, \absb{\setb{x \in B_1}{f_j^-(x) > \delta}} = 0  \qquad\text{for every $\delta>0$;}
\]
\item the family of negative parts $\{f_j^{-}\}$ is equiintegrable.
\end{enumerate}
Then, $f_j \to 0$ in $\Lrm^1_\loc(B_1)$.
\end{lemma}

\begin{proof}
Let $\phi \in \Crm^\infty_c(B_1;[0,1])$. Then,
\[
\int \phi \abs{f_j} \dd x =\int \phi f_j \dd x +2\int \phi f_j^- \dd x
 \leq \int \phi f_j \dd x +2\int f_j^- \dd x.
\]
The first term on the right-hand side vanishes as $j\to \infty$ by assumption~(i). Vitali's convergence theorem in conjunction with assumptions~(ii) and~(iii) further gives that the second term also tends to zero in the limit.
\end{proof}

\begin{proof}[Proof of Theorem~\ref{t:unsatisfactory}]
Let \(E\) be such that \(\I^\ell(E)=0\) and let us define 
\[ 
F:=\setBB{ x \in E }{ \mbox{\(\lambda_x := \dfrac{\d \mu}{\d|\mu|}(x)\) exists, belongs to \(\mathbb S^{d-1}\), and }\dfrac{\d \mu}{\d|\mu|}(x) \notin \Lambda^\ell_\A }. 
\]
By contradiction, let us suppose  that $|\mu|(F)>0$.  Note that, by the very definition of \(F\), for all \(x\in F\) there exists an $\ell$-dimensional plane $\tilde{\pi}_x\subset \R^d$ such that it holds that
\[
  \mathbb{A}^k(\xi)\lambda_x \neq 0  \quad\text{for all $\xi \in \tilde{\pi}_x \setminus \{0\}$.}
\]
By continuity, the same is true for all planes \(\pi'\) in a neighbourhood of \(\tilde{\pi}_x\). In particular, since by assumption \(\mathcal I^\ell (F)=0\), for every \(x\in E\)  there is an $\ell$-dimensional plane \(\pi_{x}\) such that
\begin{equation}\label{e:hausdorff measure -1}
\mathbb{A}^k(\xi)\lambda_x \neq 0 \quad\textrm{for all $\xi \in \pi_x\setminus\{0\}$}\qquad \textrm{and}\qquad \HH^\ell(\p_{\pi_x}(F))=0.
\end{equation}
Since we assume $|\mu|(F)>0$, by standard measure-theoretic arguments (see the proof of~\cite[Theorem~1.1]{De-PhilippisRindler16} for details), we can find a point  $x_0 \in F$, an $\ell$-dimensional plane \(\pi_0\), and a sequence of radii $r_j \todown 0$ with the following properties:
\begin{itemize}
\item[(b1)] \(\lambda:=\dfrac{\d \mu}{\d|\mu|}(x_0)\) exists, belongs to \(\mathbb S^{m-1}\), and satisfies
\begin{equation}\label{e:hausdorff measure 0}
\mathbb{A}^k(\xi)\lambda \neq 0 \qquad\textrm{for all $\xi \in \pi_0\setminus\{0\}$;}
\end{equation}
\item[(b2)] setting \(\tilde\mu^s:=\mu\res F\),
\[
\lim_{j\to\infty} \frac{ \abs{\tilde \mu^s}(B_{2r_j}(x_0))}{\abs{\mu}(B_{2r_j}(x_0))}=1\qquad \mbox{and} \qquad  \lim_{j\to \infty} \dashint_{B_{2r_j}(x_0)}  \biggl|\frac{\d\mu}{\d|\mu|} - \lambda\biggr| \dd \abs{\mu}= 0;
\]
\item[(b3)] for
\[
  \mu_j:=\frac{T^{x_0,r_j}_\#\mu}{|\mu|(B_{2r_j}(x_0))}
\]
the following convergence holds: 
\[
|\mu_j|:=\frac{T^{x_0,r_j}_\#|\mu|}{|\mu|(B_{2r_j}(x_0))}\overset{*}{\rightharpoonup} \sigma \]
 for some $\sigma \in \Mcal^+(B_2)$ with $\sigma\res{B_{1/2}}\neq 0$. Here, $T^{x_0,r_j}(x):=\dfrac{x-x_0}{r_j}$.
\end{itemize}

After a rotation we may assume that $\pi_{0}=\R^\ell \times \{0\}$. We shall use the coordinates $(y,z) \in \R^\ell\times \R^{d-\ell}$ and we will denote by $\p$ the orthogonal projection onto $\R^\ell$.
Note that 
\[
\A^k\mu_j=R_j \qquad\text{in the sense of distributions},
\]
where \(\A^k\) is the $k^\text{\tiny{th}}$-order homogeneous part of \(\A\), i.e.,
\[
\A^k:=\sum_{|\alpha|=k} A_\alpha \partial^\alpha,
\]
and \(R_j\) contains all derivatives of $\mu_j$ of order  at most $k-1$.
Thus,
\begin{equation}\label{eq:negativesob}
\text{$\{R_j\}$ is pre-compact in \(\Wrm_{\rm loc}^{-k,q}(\R^{d})\) for $1 < q < d/(d-1)$,}
\end{equation}
where \( \Wrm_{\rm loc}^{-k,q}(\R^{d})\) is the local version of the dual of the Sobolev space  \(\Wrm^{k,q'}(\R^{d})\), \(q'=q/(q-1)\).

Define
\[
\B:=\A^k\res \pi_{0}:=\sum_{\substack{\abs{\alpha}=k\\ \alpha_i =0 \text{ for } i \ge \ell+1  }} A_\alpha \partial^\alpha.
\]
Note that \(\B\) is  a homogeneous constant-coefficient linear differential operator such that  for any \(\psi \in \Crm^\infty(\R^\ell)\), 
\begin{equation}\label{e:freez}
(\B\psi)(\p x)=\A^k(\psi\circ \p)(x), \qquad x \in \R^d,
\end{equation}
and, by~\eqref{e:hausdorff measure 0},
\[
\lambda \notin \ker \bB(\xi)\qquad\text{for all $\xi \in \R^\ell\setminus \{0\}$.}
\]
Moreover, the measure
 \begin{equation*}
 \tilde \mu_j^s:=\frac{T^{x_0,r_j}_\#\tilde \mu^s}{|\mu|(B_{2r_j}(x_0))}
 \end{equation*}
is concentrated on the the set \(F_j:=T^{x_0,r_j}(F)\), which by~\eqref{e:hausdorff measure -1} satisfies
\begin{equation}\label{e:sing}
\HH^\ell(\p(F_j))=0.
\end{equation}

We consider the (localized) sequence of measures 
\[ 
\nu_j:= \p_\# (\chi \mu_j) \in \Mcal(B_2^\ell),
\] 
where \(\chi(y,z)=\tilde \chi(z)\) for some cut-off function \(\tilde \chi \in \Crm^\infty_c (B_1^{d-\ell};[0,1])\) satisfying $\chi \equiv 1$ on $B^{d-\ell}_{1/2}$. Our goal is to apply Lemma~\ref{l:strong constancy lemma} to the sequence $\{\nu_j\}\subset \mathcal{M}(B_1^\ell;\R^m)$, from where we will reach a contradiction. We must first check that $\{\nu_j\}$ satisfies the assumptions of Lemma~\ref{l:strong constancy lemma}. Since 
\[
|\nu_j|(B^\ell_1) \le |\chi\mu_j|(B_2) \le 1, 
\]
the sequence is equi-bounded.   We further claim that (b2) implies that 
\begin{equation}\label{e:from mu to nu}
\lim_{j\to \infty} \big||\nu_j|-\p_\# (\chi |\mu_j|) \big|(B^\ell_1) =0\quad\text{ and }\quad \lim_{j\to \infty} \int_{B^\ell_1} \biggl|\frac{\d\nu_j}{\d|\nu_j|}- \lambda\biggr|  \dd \nu_j =0.	
\end{equation}
Consequently, assumption~(a3) in Lemma~\ref{l:strong constancy lemma} is then satisfied for $\{\nu_j\}$.

Concerning the assumption~(a2), we argue as follows. Let $\psi \in \Crm^\infty_c(B^\ell_1;\R^n)$. Then, for the adjoint
\[
  \B^*:=(-1)^{k}\sum_{\substack{\alpha \in (\N \cup \{0\})^\ell\\ |\alpha| = k}} A_\alpha^* \, \partial^\alpha,
\]
equation~\eqref{e:freez} gives 
\begin{align*}
\int \B^*\psi  \dd \nu_j & = \int (\A^k)^* (\psi \circ \p)(y) \, \chi(z)   \dd \mu_j(y,z)\\
&=\int (\A^k)^*(\chi (\psi \circ \p)) - [(\A^k)^*,\chi](\psi \circ \p) \dd \mu_j	\\
&= \bigl\langle \chi R_{j},\psi\circ \p\bigr\rangle+ \sum_{\substack{\beta \in (\N \cup \{0\})^\ell\\ |\beta|<k}}  \int  \partial^\beta \psi(y) \, \, C_\beta(z)  \dd \mu_j(y,z),
\end{align*}
where $[(\A^k)^*,\chi] = (\A^k)^* \circ \chi - \chi \circ (\A^k)^*$ is the commutator of $(\A^k)^*$ and $\chi$, as well as $C_\beta\in \Crm^\infty_c(B_1^{d-\ell}) $.
Hence, in the sense of distributions,  
\[
 \B\nu_j = \p_\# (\chi R_j)+ \sum_{\substack{\beta \in (\N \cup \{0\})^\ell \\ |\beta|<k}} (-1)^{|\beta|}\partial^\beta \p_\# (C_\beta\mu_j). 
 \]
Note that \(\chi R_j\) is compactly supported in the \(z\)-direction and thus the push-forward under \(\p\) is well defined.  Exactly as in the proof of~\cite[Theorem~1.1]{De-PhilippisRindler16} we infer the following: since for each $\beta$ we have $\p_\# (C_\beta\mu_j) \in \mathcal{M}(B_1^\ell;\R^m)$ and $|\beta|<k$, the family \(\{(\Id-\Delta)^{-\frac s2}\partial^\beta \p_\# (C_\beta\mu_j)\}_j\) is pre-compact in \(\L^1_{\rm loc}(\R^\ell)\) for every \(s\in (k-1,k)\), and by~\eqref{eq:negativesob} the same holds for  \(\{(\Id-\Delta)^{-\frac s2}\p_\# (\chi R_j)\}_j\).
 
Thus,  we can apply Lemma~\ref{l:strong constancy lemma} to deduce (up to taking a subsequence) that
\[
 \lim_{j \to \infty}\big||\nu_j |- \theta \LL^\ell\big|(B^\ell_{1/2})=0
 \]
for some \(\theta \in \L^1(B^\ell_1)\). Consequently,
\begin{align*}
\sigma(B_{1/2}) &\!\!\stackrel{\textrm{(b3)}}\le \liminf_{j \to \infty} \abs{\mu_j}(B_{1/2})\\
& \!\!\stackrel{\textrm{(b2)}}= \liminf_{j \to \infty} |\tilde\mu_j^s|(B_{1/2})
\\
&= \liminf_{j \to \infty} |\mu_j|(B_{1/2}\cap F_j)\\
&\le \liminf_{j \to \infty} \big|\p_\#(\chi |\mu_j|)\big|(B^\ell_{1/2}\cap \p(F_j))
\\
&\!\!\stackrel{\eqref{e:from mu to nu}}\le\liminf_{j \to \infty} |\nu_j|(B^\ell_{1/2}\cap \p(F_j))\\
&\le \int_{B^\ell_{1/2}\cap \p(F_j)} \theta  \dd \LL^\ell + \lim_{j \to \infty}\big||\nu_j| - \theta \d\LL\big|(B^\ell_{1/2}) \\
&\!\!\stackrel{\eqref{e:sing}}= 0.
\end{align*}
However, $\sigma(B_{1/2}) = 0$ is a contradiction to~(b3).

It remains to show the claim~\eqref{e:from mu to nu}. By disintegration, see for instance~\cite[Theorem~2.28]{AmbrosioFuscoPallara00},  for every  $j\in \N$, 
\[ 
\chi |\mu_j| = \nu^j_y \otimes \kappa_j \quad\text{with}\quad \kappa_j = \p_\# (\chi |\mu_j|). 
\]
Here, each $\nu_y^j$ is a  probability measure supported in \(B_{1}^{d-\ell}\). Let
\[
  f_j(y,z):= \frac{\d\mu_j}{\d|\mu_j|}(y,z).
\]
Then, 
\[  
\p_\# (\chi \mu_j) = g_j(y) \kappa_j(\di y) = \nu_j  \quad\text{with}\quad g_j(y):=\int_{\R^{d-\ell}} f_j(y,z)  \dd \nu^j_y(z). 
\]
 In particular, $|g_j| \le 1$. Furthermore, since $|\lambda|=1$,   
\begin{align*} 
0 &\le \int_{B^{\ell}_1} (1 - |g_j(y)|) \dd \kappa_j(y) \\
&= \int_{B_1^{\ell}} \biggl| \int_{B_{1}^{d-\ell}} \lambda  \dd \nu^j_y(z) \biggr|\dd \kappa_j(y)- \int_{B_1^{\ell}}\biggl|\int_{B_{1}^{d-\ell}} f_j(y,z)  \dd \nu^j_y(z)\biggr| \dd \kappa_j(y) \\
&\le \int_{B_1^\ell\times B_{1}^{d-\ell}} \abs{f_j-\lambda}  \dd (\nu^j_y\otimes\kappa_j) \\
&\le \int_{B_2} |f_j -\lambda|  \dd |\mu_j|  \stackrel{\textrm{(b2)}}\to 0 \quad\textrm{as \(j\to \infty\)}.
\end{align*} 
Since $\abs{\nu_j}= \abs{g_j} \kappa_j$, this proves the first part of~\eqref{e:from mu to nu}. The second part follows from this estimate and (b2) because
 \begin{align*}
 \int_{B^{\ell}_1} \biggl|\frac{g_j}{|g_j|} - \lambda\biggr|  \dd |\nu_j|
 &= \int_{B^{\ell}_1} \big|g_j - |g_j|\lambda\big| \dd\kappa_j \\
 &\le \int_{B^{\ell}_1} |g_j -\lambda| \dd\kappa_j+ \int_{B^{\ell}_1} (1- |g_j|) \dd\kappa_j \\
&\le \int_{B_2} | f_j - \lambda| \d|\mu_j| + \int_{B^{\ell}_1} (1- |g_j|) \dd\kappa_j  \to 0  \quad\textrm{as \(j\to \infty\).}
\end{align*}
This concludes the proof.
\end{proof}

Before proving Theorem~\ref{t:rect}, let us start with the following elementary lemma:
\begin{lemma}\label{lm:rec}
Let \(\mu\) be an \(\A\)-free measure and assume that there exists an \(\ell\)-rectifiable set \(R\)  such that 
\begin{equation}\label{aac1}
\HH^{\ell}\res R \ll |\mu|\res R\ll \HH^{\ell}\res R.
\end{equation}
Then,
\begin{equation}\label{mutual_abs_cont1}
\mu\res R=\theta^*_{\ell}(|\mu|)\lambda\,\HH^\ell\res R,
\end{equation}
where \(\lambda \colon R \to \mathbb S^{m-1}\) is $\HH^\ell$-measurable. Moreover  for \(\HH^\ell\)-almost every \(x_0\in R\),
\begin{equation}\label{bu1}
(2r)^{-\ell}(T^{x_0,r})_\#\mu\overset{*}{\rightharpoonup} \theta^{*}_{\ell}(|\mu|)(x_0) \lambda(x_0)\HH^\ell\res (T_{x_0} R)  \qquad
\text{as $r \todown 0$,}
\end{equation}
and
\[
\lambda(x_0)\in \bigcap_{\xi \in (T_{x_0} R)^\perp} \ker \bA^k(\xi),
\]
where  \(T_{x_0} R\) is the the approximate tangent plane to \(R\) at \(x_0\).
\end{lemma}

\begin{proof}
By~\cite[Theorem~6.9]{Mattila95},
\[
\HH^\ell\big(\{\theta_{\ell}^*(|\mu|)=+\infty\}\big)=0.
\]
Hence, by~\eqref{aac1}, we can assume that \(R\subset  \{\theta_{\ell}^*(|\mu|)<+\infty\}\). In particular, by~\cite[Theorem~6.9]{Mattila95} again,  \(\HH^{\ell}\res R\) is \(\sigma\)-finite and the Radon--Nikod\'{y}m theorem implies
\[
\mu\res R= f \HH^{\ell}\res R 
\]
with  \(f\in \L^1(R,\HH^\ell;\R^m)\) such that \(|f|>0\) \((\HH^{\ell}\res R)\)-almost everywhere. A standard blow-up argument then gives~\eqref{mutual_abs_cont1} and~\eqref{bu1}. Choosing a point such that the conclusion of~\eqref{bu1} holds true and blowing up around that point, one deduces that the measure 
\[
\bar{\mu}:=\lambda (x_0) \, \HH^\ell\res (T_{x_0} R)
\] 
is \(\A^k\)-free, where \(\A^k\) is the \(k\)-homogeneous part of \(\A\). Since \(\HH^\ell\res (T_{x_0} R)\) is a tempered distribution, by taking the Fourier transform of the equation \(\A^k\bar{\mu}=0\), we obtain
\[
\bA^k(\xi) \lambda(x_0) \, \HH^{d-\ell}\res (T_{x_0} R)^\perp =0,
\]
which implies that \(\bA^k(\xi) \lambda(x_0)=0\) for all \(\xi \in (T_{x_0} R)^\perp\). This concludes the proof.
\end{proof}

\begin{proof}[Proof of Theorem~\ref{t:rect}.]
By classical measure theory, see~\cite[Theorem~6.9]{Mattila95},
\[
\HH^\ell\big(\{\theta_{\ell}^*(|\mu|)=+\infty\}\big)=0.
\]
Hence, the assumption \(\Lambda_{\A}^\ell=\{0\}\) and Corollary~\ref{cor:aac} together imply that
\[
  |\mu|\big(\{\theta_{\ell}^*(|\mu|)=+\infty\}\big) = 0.
\]
By~\cite[Theorem~6.9]{Mattila95}, the set 
\[
G:=\{\theta^*_{\ell}(\mu)\in (0,+\infty)\}
\]
is \(\HH^\ell\) \(\sigma\)-finite and 
\begin{equation}\label{abs}
|\mu|\res G\ll\HH^\ell\res G\ll|\mu|\res G.
\end{equation}
According to~\cite[Theorem~15.6]{Mattila95} we may write 
\[
G=R\cup S,
\]
where \(R\) is \(\HH^\ell\)-rectifiable, \(S\) is purely unrectifiable and \(\HH^\ell(R\cap S)=0\). By the Besicovitch--Federer rectifiability theorem, see~\cite[Section~3.3.13]{Federer69},~\cite[Chapter~18]{Mattila95} or~\cite{White98},
\[
  \I^{\ell}(S)=0.
\]
Hence, since \(\Lambda_{A}^\ell=\{0\}\), Corollary~\ref{cor:aac} implies that \(|\mu|(S)=0\). Therefore,
\[
\mu\res\{\theta^*_{\ell}(|\mu|)>0\}=\mu\res G=\mu\res R.
\]
Owing to this and to \eqref{abs} we can apply  Lemma~\ref{lm:rec}  and thus conclude the proof.
\end{proof}

\section{Applications}\label{sec:applications}

In this section~we sketch applications of the abstract results to several common differential operators $\A$. In this way we recover and improve several known results.

\subsection{Rectifiability of \(\BV\)-gradients} Let \(\mu=Du\in \mathcal M(U;\R^p\otimes \R^d)\), where \(u\in \BV (U; \R^p)\), $U \subset \R^d$ open; see~\cite{AmbrosioFuscoPallara00} for details on this space of functions of bounded variation. Then \(\mu\) is \(\curl\)-free. By a direct computation,
\[
\ker (\curl )(\xi)=\setb{ a\otimes \xi }{ a\in \R^p, \xi\in \R^d },  \qquad \xi \in \R^d,
\]
hence \(\Lambda^{d-1}_{\curl}=\{0\}\) and Corollary~\ref{cor:aac} in conjunction with Theorem~\ref{t:rect} implies the well-known fact that $|Du|\ll \HH^{d-1}$ and 
\[
Du\res\{\theta_{d-1}^*(|Du|)>0\}=  a(x)\otimes n_R(x) \, \HH^{d-1}_x\res R
\]
for some \((d-1)\)-rectifiable set \(R \subset U\) and where $ n_R : R \to \Sbb^{d-1}$ is a measurable map with the property that $ n_R(x)$ is orthogonal to $T_xR$ at $\Hcal^{d-1}$-almost every $x$. This is the well-known rectifiability result of $\BV$-maps (see~\cite{AmbrosioFuscoPallara00}). 

\subsection{Rectifiability of symmetrized  gradients} Let $U \subset \R^d$ be an open set and let \(\mu=Eu\in \mathcal M(U, (\R^d\otimes\R^d)_{\rm sym})\), where \(u\in \BD (U; \R^d)\) is a function of bounded deformation and 
\[
 Eu \coloneqq \frac{Du+Du^{T}}{2}
\]
is the symmetric part of the distributional derivative of $u$. Then \(\mu\) is \(\curl \curl\)-free (see~\cite[Example~3.10(e)]{FonsecaMuller99}), where
 \[
\curl\curl  \mu:=\sum_{i=1}^d \partial_{ik} \mu_{i}^j+\partial_{ij} \mu_{i}^k-\partial_{jk} \mu_{i}^i-\partial_{ii} \mu_{j}^k, \qquad j,k=1,\ldots,d\,.
\] 
In this case,
\[
\ker (\curl \curl) (\xi)=\setb{ a\odot \xi }{ a\in \R^d, \xi \in \R^d },\qquad \xi \in \R^d,
\]
where $a\odot \xi := (a \otimes \xi + \xi \otimes a)/2$. Hence, \(\Lambda^{d-1}_{\curl\curl }=\{0\}\).   Corollary~\ref{cor:aac} and Theorem~\ref{t:rect} yield that $|Eu| \ll \Hcal^{d-1}$ and
\[
Eu\res\{\theta_{d-1}^*(|Eu|)>0\}= a(x)\odot n_R(x) \, \HH^{d-1}_x\res R,
\]
for some \((d-1)\)-rectifiable set \(R \subset U\) and $ n_R(x)$ is orthogonal to $T_xR$ at $\Hcal^{d-1}$ almost every $x$. This comprises the dimensional estimates and rectifiability of $\BD$-functions from~\cite{Kohn79,AmbrosioCosciaDalMaso97} (see in particular~\cite[Proposition~3.5]{AmbrosioCosciaDalMaso97}).

\subsection{Rectifiability of  varifolds and  defect measures} Let $U \subset \R^d$ be an open set and let us  assume that \(\boldsymbol{\mu} \in \mathcal M(U;\R^d\otimes \R^d)\) is a matrix-valued measure satisfying 
\[
\Div \boldsymbol \mu=\sigma \in \mathcal M(U;\R^d),
\]
where ``\(\Div\)'' is the row-wise divergence.

\begin{proposition}\label{prop:div}
Let \(\boldsymbol{\mu} \in \mathcal M(U;\R^d\otimes \R^d)\) be as above. Assume that for \(|\boldsymbol{\mu}|\)-almost every \(x \in U\),
\[
\rank \biggl( \frac{\d \boldsymbol{\mu}}{\d |\boldsymbol{\mu}|}(x)\biggr)\ge \ell.
\]
Then, \(|\boldsymbol{\mu}|\ll\Ical^\ell\ll \HH^{\ell}\) and there exists an $\ell$-rectifiable set $R\subset U$ and a $\HH^\ell$-measurable map $\lambda \colon R \to \R^d \otimes \R^d$ satisfying
\[
\rank \lambda(x) = \ell \quad \text{$\HH^\ell$-almost everywhere},
\]
such that
\[
\mu\res \{\theta^{*}_{\ell}(|\boldsymbol{\mu}|)>0\} = \lambda(x) \, \HH^\ell_x \res R.
\] 
\end{proposition}

\begin{proof}
Let \(\widetilde{\boldsymbol{\mu}}:=(\boldsymbol{\mu},\sigma)\in \mathcal M(U;(\R^d\otimes \R^d)\times \R^d) \) and let us define the (non-homogeneous) operator \(\A\) via
\[
\A \widetilde{\boldsymbol{\mu}} :=\Div \boldsymbol{\mu}-\sigma,
\]
so that  \(\ker\bA^1(\xi)= \ker (\Div) (\xi)\times \R^d\). Since 
\[
 \ker (\Div) (\xi)=\setb{M \in \R^d\otimes \R^d}{ \xi \in \ker M}, 
\]
we see that
\[
\begin{split}
\Lambda_{\A}^{\ell}&=\bigcap_{\pi \in \Gr(\ell, d)}\setb{M \in \R^d\otimes \R^d}{ \ker M\cap \pi \ne \{0\} }\times \R^d
\\
&= \setb{ M \in \R^d\otimes \R^d }{{\rm dim} \ker M >d-\ell}\times \R^d.
\end{split}
\]  
Since $|\boldsymbol{\mu}|\ll |\widetilde{\boldsymbol{\mu}}|$, for \(|\boldsymbol{\mu}|\)-almost every $x$ there exists a scalar \(\tau(x)\ne 0\) such that 
\[
\frac{\d \boldsymbol{\mu}}{\d |\boldsymbol{\mu}|}(x)=\tau(x)\frac{\d \boldsymbol{\mu}}{\d |\widetilde{\boldsymbol{\mu}}|}(x),
\]
and hence by Theorem~\ref{t:unsatisfactory},
\[
\text{$\Ical^\ell(B)=0$ for $B$ Borel } \quad \Longrightarrow \quad \rank\bigg(\frac{\d \boldsymbol{\mu}}{\d |\boldsymbol{\mu}|}(x)\bigg) < \ell \quad \text{for $|\boldsymbol{\mu}|$-a.e. $x \in B$}.
\]
In particular, by the assumption on the lower bound of the rank, we deduce that $|\boldsymbol{\mu}|\ll \Ical^\ell\ll\HH^\ell$ and that there exists a rectifiable set \(R\) such that 
\[
\abs{\boldsymbol{\mu}} \res \{\theta^*_\ell(|\boldsymbol{\mu}|)> 0 \} =  \HH^\ell \res R.
\]
The last part of the theorem then easily follows from Lemma \ref{lm:rec}.
\end{proof}

The above proposition allows, for instance,  to  reprove the results of~\cite{Allard72} and to slightly improve the one in~\cite{De-PhilippisDe-RosaGhiraldin18}. To see this, recall that an \(\ell\)-dimensional  varifold can be seen as a measure \(V\) on \(\R^d\times \Gr(\ell,d)\) and that the condition of having bounded first variation with respect to an integrand \(F\) can be written as 
\[
\Div \big(A_F (V_{x}) \|V\|\big)\in \mathcal M(\R^d;\R^d),
\]
where \(\|V\|\) is the projection of \(V\) on \(\R^d\) (the first factor),  \(V(\d x,\d T)=V_{x}(\d T)\otimes \|V\|(\d x)\) is the disintegration of \(V\) with respect to this projection,
\[
A_F (V_{x}):=\int_{\Gr(\ell,d)} B_F(x,T) \dd V_{x}(T) \quad\in \R^d\times \R^d,
\]
and  \(B_F \colon \R^d\times \Gr(\ell,d)\to \R^d\otimes \R^d\) is a matrix-valued map that  depends on the specific integrand \(F\), see the introduction of~\cite{De-PhilippisDe-RosaGhiraldin18} for details. 

The (AC)-condition in~\cite[Definition 1.1]{De-PhilippisDe-RosaGhiraldin18}  exactly implies that the assumptions of Proposition~\ref{prop:div} are satisfied. We remark that in fact Proposition~\ref{prop:div} allows to slightly improve~\cite[Theorem~1.2]{De-PhilippisDe-RosaGhiraldin18}  in the following respects:
\begin{itemize}
\item[(a)] One obtains that \(V\res\{\theta^{*}_{\ell}>0\}\) is rectifiable while in~\cite{De-PhilippisDe-RosaGhiraldin18} only the rectifiability of \(V\res\{\theta_{*, \ell}>0\}\) is shown (here, $\theta_{*,\ell}$ is the lower $\ell$-dimensional Hausdorff density map). 
\item[(b)] If one only wants to get the rectifiability of the measure  \(\|V\|\res\{\theta^{*}_{\ell}>0\}\), then condition~(i) in~\cite[Definition 1.1]{De-PhilippisDe-RosaGhiraldin18} is enough. This allows, in the case \(\ell=d-1\), to work with convex but not necessarily strictly convex integrands. 
\end{itemize}

By similar arguments one recovers the results of \textsc{Ambrosio \& Soner}~\cite{AmbrosioSoner97}, and of \textsc{Lin}~\cite{Lin99} and \textsc{Moser}~\cite{Moser03}; we omit the details.


\begin{thebibliography}{10}

\bibitem{Alberti93}
{\sc G.~Alberti}, {\em Rank one property for derivatives of functions with
  bounded variation}, Proc. Roy. Soc. Edinburgh Sect. A, 123 (1993),
  pp.~239--274.

\bibitem{AlbertiMarchese16}
{\sc G.~Alberti and A.~Marchese}, {\em On the differentiability of {L}ipschitz
  functions with respect to measures in the {Euclidean} space}, Geom. Funct.
  Anal., 26 (2016), pp.~1--66.

\bibitem{AlbertiMassaccesiStepanov}
{\sc G.~Alberti, A.~Massaccesi, and E.~Stepanov}, {\em On the geometric
  structure of normal and integral currents, {II}}.
\newblock In preparation.

\bibitem{Allard72}
{\sc W.~K. Allard}, {\em {On the first variation of a varifold}}, Ann. of
  Math., 95 (1972), pp.~417--491.

\bibitem{Allard86}
\leavevmode\vrule height 2pt depth -1.6pt width 23pt, {\em {An integrality
  theorem and a regularity theorem for surfaces whose first variation with
  respect to a parametric elliptic integrand is controlled}}, in {Geometric
  Measure Theory and the Calculus of Variations}, W.~K. Allard and F.~J.
  {Almgren Jr.}, eds., vol.~44 of Proceedings of Symposia in Pure Mathematics,
  Amer. Math. Soc., 1986.

\bibitem{AmbrosioCosciaDalMaso97}
{\sc L.~Ambrosio, A.~Coscia, and G.~{Dal~Maso}}, {\em Fine properties of
  functions with bounded deformation}, Arch. Ration. Mech. Anal., 139 (1997),
  pp.~201--238.

\bibitem{AmbrosioFuscoPallara00}
{\sc L.~Ambrosio, N.~Fusco, and D.~Pallara}, {\em {Functions of bounded
  variation and free discontinuity problems}}, {Oxford Mathematical
  Monographs}, Oxford University Press, 2000.

\bibitem{AmbrosioSoner97}
{\sc L.~Ambrosio and H.~M. Soner}, {\em A measure-theoretic approach to higher
  codimension mean curvature flows}, Ann. Scuola Norm. Sup. Pisa Cl. Sci. (4),
  25 (1997), pp.~27--49 (1998).

\bibitem{Arroyo-Rabasa18}
{\sc A.~{Arroyo-Rabasa}}, {\em {An elementary approach to the dimension of
  measures satisfying a first-order linear PDE constraint}}.
\newblock arXiv:1812.07629, 2018.

\bibitem{AyoushWojciechowski17?}
{\sc R.~Ayoush and M.~Wojciechowski}, {\em On dimension and regularity of
  bundle measures}.
\newblock arXiv:1708.01458, 2017.

\bibitem{Bate17}
{\sc D.~{Bate}}, {\em {Purely unrectifiable metric spaces and perturbations of
  Lipschitz functions}}.
\newblock arXiv:1712.07139, 2017.

\bibitem{De-Philippis16}
{\sc G.~De~Philippis}, {\em On the singular part of measures constrained by
  linear {PDE}s and applications}, in {Proceedings of the Seventh European
  Congress of Mathematics (Berlin 2016)}, 2016.

\bibitem{De-PhilippisDe-RosaGhiraldin18}
{\sc G.~De~Philippis, A.~De~Rosa, and F.~Ghiraldin}, {\em Rectifiability of
  varifolds with locally bounded first variation with respect to anisotropic
  surface energies}, Comm. Pure Appl. Math., 71 (2018), pp.~1123--1148.

\bibitem{De-PhilippisMarcheseRindler17}
{\sc G.~{De Philippis}, A.~Marchese, and F.~Rindler}, {\em On a conjecture of
  {C}heeger}, in Measure Theory in Non-Smooth Spaces, N.~Gigli, ed., De
  Gruyter, 2017, pp.~145--155.

\bibitem{De-PhilippisRindler17a}
{\sc G.~{De Philippis} and F.~{Rindler}}, {\em {On the structure of measures
  constrained by linear PDEs}}, in {Proceedings of the International Congress
  of Mathematicians 2018 (Rio de Janeiro 2018)}.

\bibitem{De-PhilippisRindler16}
{\sc G.~De~Philippis and F.~Rindler}, {\em {On the structure of
  $\mathcal{A}$-free measures and applications}}, Ann. of Math., 184 (2016),
  pp.~1017--1039.

\bibitem{Federer69}
{\sc H.~Federer}, {\em {Geometric measure theory}}, {Die Grundlehren der
  mathematischen Wissenschaften, Band 153}, Springer, 1969.

\bibitem{FonsecaMuller99}
{\sc I.~Fonseca and S.~M{\"u}ller}, {\em {$\mathcal{A}$}-quasiconvexity, lower
  semicontinuity, and {Y}oung measures}, SIAM J. Math. Anal., 30 (1999),
  pp.~1355--1390.

\bibitem{GigliPasqualetto16}
{\sc N.~Gigli and E.~Pasqualetto}, {\em {Behaviour of the reference measure on
  {$\sf RCD$} spaces under charts}}.
\newblock arXiv:1607.05188, 2016.

\bibitem{Grafakos14}
{\sc L.~Grafakos}, {\em Classical {F}ourier analysis}, vol.~249 of Graduate
  Texts in Mathematics, Springer, New York, 3rd~ed., 2014.

\bibitem{KellMondino16}
{\sc M.~Kell and A.~Mondino}, {\em On the volume measure of non-smooth spaces
  with {R}icci curvature bounded below}, Ann. Sc. Norm. Super. Pisa Cl. Sci.,
  18 (2018), pp.~593--610.

\bibitem{Kohn79}
{\sc R.~V. Kohn}, {\em New estimates for deformations in terms of their
  strains}, PhD thesis, Princeton University, 1979.

\bibitem{Lin99}
{\sc F.-H. Lin}, {\em Gradient estimates and blow-up analysis for stationary
  harmonic maps}, Ann. of Math., 149 (1999), pp.~785--829.

\bibitem{MassaccesiVittone16}
{\sc A.~Massaccesi and D.~Vittone}, {\em An elementary proof of the rank-one
  theorem for {BV} functions}, J. Eur. Math. Soc.,  (2016).
\newblock to appear, arXiv:1601.02903.

\bibitem{Mattila95}
{\sc P.~Mattila}, {\em {Geometry of sets and measures in Euclidean spaces.
  Fractals and rectifiability}}, vol.~44 of {Cambridge Studies in Advanced
  Mathematics}, Cambridge University Press, 1995.

\bibitem{Moser03}
{\sc R.~Moser}, {\em Stationary measures and rectifiability}, Calc. Var.
  Partial Differential Equations, 17 (2003), pp.~357--368.

\bibitem{Murat81}
{\sc F.~Murat}, {\em Compacit\'e par compensation: condition n\'ecessaire et
  suffisante de continuit\'e faible sous une hypoth\`ese de rang constant},
  Ann. Sc. Norm. Sup. Pisa Cl. Sci., 8 (1981), pp.~69--102.

\bibitem{Raita18TR}
{\sc B.~Raita}, {\em {$L^1$}-estimates and $\mathbb{A}$-weakly differentiable
  functions}, Technical Report OxPDE-18/01, University of Oxford, 2018.

\bibitem{Rindler18book}
{\sc F.~Rindler}, {\em {Calculus of Variations}}, Universitext, Springer, 2018.

\bibitem{SpectorVanSchaftingen18?}
{\sc D.~Spector and J.~{Van Schaftingen}}, {\em Optimal embeddings into
  {Lorentz} spaces for some vector differential operators via {Gagliardo's}
  lemma}.
\newblock arXiv:1811.02691, 2018.

\bibitem{Stein93}
{\sc E.~M. Stein}, {\em {Harmonic Analysis}}, 1993.

\bibitem{Van-Schaftingen13}
{\sc J.~{Van Schaftingen}}, {\em Limiting {S}obolev inequalities for vector
  fields and canceling linear differential operators}, J. Eur. Math. Soc., 15
  (2013), pp.~877--921.

\bibitem{White98}
{\sc B.~White}, {\em A new proof of {F}ederer's structure theorem for
  {$k$}-dimensional subsets of {${\bf R}^N$}}, J. Amer. Math. Soc., 11 (1998),
  pp.~693--701.

\end{thebibliography}


\end{document}